\documentclass[a4paper]{amsart}
\usepackage{multicol,ifthen}
\usepackage{amsmath}
\usepackage{amssymb}
\usepackage{amsthm} 
\usepackage{enumerate}

\begin{document}

\newcommand{\F}{\mathcal{F}}
\newcommand{\Sc}{\mathcal{S}}
\newcommand{\R}{\mathbb R}
\newcommand{\T}{\mathbb T}
\newcommand{\N}{\mathbb N}
\newcommand{\Z}{\mathbb Z}
\newcommand{\C}{\mathbb C}  
\newcommand{\h}[2]{\mbox{$ \widehat{H}^{#1}_{#2}(\R)$}}
\newcommand{\hh}[3]{\mbox{$ \widehat{H}^{#1}_{#2, #3}$}} 
\newcommand{\n}[2]{\mbox{$ \| #1\| _{ #2} $}} 
\newcommand{\x}{\mbox{$X^r_{s,b}$}} 
\newcommand{\xx}{\mbox{$X_{s,b}$}}
\newcommand{\X}[3]{\mbox{$X^{#1}_{#2,#3}$}} 
\newcommand{\XX}[2]{\mbox{$X_{#1,#2}$}}
\newcommand{\q}[2]{\mbox{$ {\| #1 \|}^2_{#2} $}}
\newcommand{\e}{\varepsilon}
\newcommand{\om}{\omega}
\newcommand{\lb}{\langle}
\newcommand{\rb}{\rangle}
\newcommand{\ls}{\lesssim}
\newcommand{\gs}{\gtrsim}
\newcommand{\pd}{\partial}
\newtheorem{lemma}{Lemma} 
\newtheorem{kor}{Corollary} 
\newtheorem{theorem}{Theorem}
\newtheorem{prop}{Proposition}

\title[Modified Zakharov-Kuznetsov equation in 3 D]{A remark on the modified Zakharov-Kuznetsov equation in three space dimensions}

\author[A.~Gr\"unrock]{Axel~Gr\"unrock}

\address{Heinrich-Heine-Universit\"at D\"usseldorf, Mathematisches Institut,
Universit\"atsstra{\ss}e 1, 40225 D\"usseldorf, Germany}
\email{gruenroc@math.uni-duesseldorf.de}


\begin{abstract}
 The Cauchy Problem for the modified Zakharov-Kuznetsov equation in three space dimensions is shown to be locally
well-posed in $H^s(\R^3)$ for $s > \frac12$. Combined with the conservation of mass and energy this result implies
global well-posedness for small data in $H^1(\R^3)$.
\end{abstract}

 \maketitle

\section{Introduction and main result}

The Cauchy Problem for the modified Zakharov-Kuznetsov (mZK) equation in two space dimensions

\begin{equation}\label{mZK2}
 u_t + \pd^3_x u + \pd_x\pd^2_y u + \pd_x (u^3)=0, \qquad u(0,x,y)=u_0(x,y), \qquad (x,y)\in \R^2
\end{equation}

has been studied extensively in recent years. Local well-posedness in $H^1(\R^2)$ was obtained in 2003 by Biagioni and
Linares, see \cite{BL}. Combined with the conservation of mass and energy their local result implies global well-posedness,
provided the data are sufficiently small in $L^2(\R^2)$. The local result was generalized to data in $H^s(\R^2)$, $s>\frac34$
by Linares and Pastor in \cite{LP09}, the same authors showed global well-posedness in $H^s(\R^2)$, $s>\frac{53}{63}$, under
an additional smallness assumption on the $L^2$-norm of the data \cite{LP11}. Further progress on the local problem was reached by
Ribaud and Vento in \cite{RV12}, who established well-posedness in $H^s(\R^2)$ for $s>\frac14$. Since the mZK-equation is a
higher dimensional generalization of the modified Korteweg-de Vries equation, for which the Cauchy Problem is known to be ill-posed
in the $C^0$-uniform sense in $H^s(\R)$ with $s < \frac14$ (see \cite{KPV01}), this may be the best possible result except for the
endpoint $s=\frac14$. On the other hand the critical space obtained by scaling considerations for mZK in two dimensions is
$L^2(\R^2)$. So local well-posedness of \eqref{mZK2} in $H^s(\R^2)$ with $0 \le s \le \frac14$ ist still an open problem.

\qquad

In contrast to the two dimensional case, no results concerning the Cauchy Problem for the mZK-equation in three space dimensions
are known, yet. In this short note we shall apply mostly well-known linear and a new bilinear estimate for solutions of the
corresponding linear equation to establish the following result.

\begin{theorem}\label{LWPZK3D}
 The Cauchy-Problem for the modified Zakharov-Kuznetsov equation

\begin{equation}\label{mZK}
 u_t + \pd_x \Delta u + \pd_x (u^{3})=0, \qquad u(0,x,y)=u_0(x,y)
\end{equation}
with $x \in \R$ and $y \in \R^2$ is locally well-posed for data $u_0 \in H^s(\R^3)$, provided that $s> \frac12$.
\end{theorem}

We remark that from the scaling point of view this theorem covers the whole subcritical range. Combining the above local result
with the conservation of mass and energy as in \cite[p. 3]{FLP} we obtain global well-posedness for small data in $H^1(\R^3)$.

\begin{kor}
 Let $u_0 \in H^s(\R^3)$ with $s \ge 1$. Then there exists $\e > 0$, such that for $\|u_0\|_{H^1} < \e$ the local
solution of the Cauchy Problem \eqref{mZK} guaranteed by Theorem \ref{LWPZK3D} extends globally in time.
\end{kor}

\section{Estimates for free solutions of the linear equation} 

Let $U_{\phi}(t)u_0$ denote the solution of the Cauchy Problem for the linear equation

\begin{equation}\label{ZKlin}
 u_t + \pd_x \Delta u=0, \qquad u(0,x,y)=u_0(x,y),
\end{equation}

where $t \in \R$, $x \in \R$, $y \in \R^2$ and $\Delta = \pd_x^2 + \Delta_y$, $\Delta_y=\pd_{y_1}^2+\pd_{y_2}^2$.
The phase function is given here by $\phi(\xi,\eta)=\xi(\xi^2+|\eta|^2)$, where $(\xi,\eta) \in \R\times\R^2$
are the dual variables corresponding to $(x,y)\in \R\times\R^2$. Then the Strichartz type estimate 

\begin{equation}\label{ZK.3}
 \|D_x^{\theta \e /2}U_{\phi}u_0\|_{L_t^pL_{xy}^q}\ls \|u_0\|_{L^2_{xy}}
\end{equation}

is valid for $0 < \e < 1$, $0<\theta< (1+\frac{\e}{3})^{-1}$, $\frac{2}{p}=\theta(1+\frac{\e}{3})$
and $\frac{1}{q}=\frac{1-\theta}{2}$. It was obtained by Linares and Saut, see \cite[Proposition 3.1]{LS}.
The $L^4$-estimate corresponding to $\theta= \frac12$ and $\e=0$ is excluded in \eqref{ZK.3} but
nonetheless true. In fact, modifying the proof of Theorem 2 in \cite{GPS} appropriately we obtain
the bilinear estimate

\begin{equation}\label{ZK.4}
 \|U_{\phi}u_0U_{\phi}v_0\|_{L^2_{txy}} \ls \|D_x^{-\frac{1}{2}}u_0\|_{L^2_{xy}}\|\lb D_x\rb^sv_0\|_{L^2_{xy}},
\end{equation}

provided $s>\frac12$. Especially for $u_0=v_0$ we get with $P_{\Delta_k}=\F_x^{-1}\chi_{\{\xi \sim 2^k\}}\F_x$ that

$$\|P_{\Delta_1}U_{\phi}u_0\|_{L^4_{txy}} \ls \|P_{\Delta_1}u_0\|_{L^2_{xy}}.$$

Rescaling we see that also

$$\|P_{\Delta_k}U_{\phi}u_0\|_{L^4_{txy}} \ls \|P_{\Delta_k}u_0\|_{L^2_{xy}}$$

and hence by the Littlewood Paley Theorem

\begin{equation}\label{ZK.5}
 \|U_{\phi}u_0\|_{L^4_{txy}} \ls \|u_0\|_{L^2_{xy}}.
\end{equation}

Apart from the Strichartz type estimates and their bilinear refinement we can rely on a local smoothing
effect of Kato type in order to deal with the derivative in the nonlinearity. As was shown by Ribaud and
Vento, in the case of the linear ZK equation \eqref{ZKlin} it reads

\begin{equation}\label{ZK.6}
 \|\nabla_{xy}U_{\phi}u_0\|_{L_x^{\infty}L^2_{yt}} \ls \|u_0\|_{L^2_{xy}},
\end{equation}

see Proposition 3.1 of \cite{RV11}. To complement the use of the local smoothing effect we combine a Sobolev
embedding in the time variable with the $L^4$-Strichartz estimate in order to obtain the following maximal function
inequality. The argument given below was taken from \cite[Proof of Theorem 2.4]{KPV91}.

\begin{eqnarray}\label{ZK.7}
 \|U_{\phi}u_0\|_{L^4_{xy}L^{\infty}_t} & \ls & \|\lb D_t\rb ^{\frac{1}{4}+}U_{\phi}u_0\|_{L^4_{xyt}} \nonumber \\
                                 & = & \|\lb\pd_x\Delta\rb^{\frac{1}{4}+}U_{\phi}u_0\|_{L^4_{xyt}} \ls \|u_0\|_{{H}^s}, \qquad s>\frac{3}{4}.
\end{eqnarray}

\section{Proof of the local result for modified ZK}

Now let $X_{s,b}$ denote the Bourgain space associated with the phase function
$\phi(\xi,\eta)=\xi(\xi^2+|\eta|^2)$, more precisely let

$$X_{s,b} = \{u \in \Sc '(\R^4): \|u\|_{X_{s,b}}<\infty\},$$

where, with $(\xi,\eta,\tau)\in \R \times \R^2 \times \R$,

$$\|u\|_{X_{s,b}}= \|\lb(\xi,\eta)\rb^s\lb\tau - \phi(\xi,\eta)\widehat{u}\rb^b\|_{L^2_{\xi \eta \tau}}.$$

Then by the transfer principle the estimates for free solutions discussed in Section 2 imply corresponding estimates in
$X_{s,b}$ - norms for $b > \frac12$. For example we have 

\begin{equation}\label{ZK.10}
 \|u\|_{L^4_{txy}} \ls \|u\|_{X_{0,b}},
\end{equation}

\begin{equation}\label{ZK.11}
 \|\nabla_{xy}u\|_{L_x^{\infty}L^2_{yt}} \ls \|u\|_{X_{0,b}},
\end{equation}

and, for $s > \frac34 $,

\begin{equation}\label{ZK.12}
 \|u\|_{L^4_{xy}L^{\infty}_t}  \ls \|u\|_{X_{s,b}}.
\end{equation}

The bilinear estimate \eqref{ZK.4} is converted into

\begin{equation}\label{ZK.13}
 \|uv\|_{L^2_{xyt}}  \ls \|D_x^{-\frac12}u\|_{X_{0,b}}\|\lb D_x \rb^s v\|_{X_{0,b}},
\end{equation}

where $s> \frac12$ and again $b> \frac12$ are assumed. In the sequel we proceed similar as in
\cite[Proof of Theorem 2]{G09} and combine these estimates with duality and interpolation arguments
to obtain the following Proposition, which in turn implies Theorem \ref{LWPZK3D}.

\begin{prop}\label{mZKprop}
 For any $s > \frac12$ there exists a $b' > - \frac12$, such that for all $b > \frac12$ the estimate

$$\|\pd_x(uvw)\|_{X_{s,b'}}\ls \|u\|_{X_{s,b}}\|v\|_{X_{s,b}}\|w\|_{X_{s,b}}$$

holds true.
\end{prop}

\begin{proof}
 Dualizing the bilinear estimate \eqref{ZK.13} we obtain for $s,b > \frac12$

\begin{equation}\label{ZK.14}
 \|D_x^{\frac12}(uv)\|_{X_{0,-b}} \ls \|u\|_{L^2_{xyt}}\|\lb D_x \rb^s v\|_{X_{0,b}}.
\end{equation}

Here and in \eqref{ZK.13} we clearly may replace the $\lb D_x \rb^s$ on the right by $\lb \nabla_{xy} \rb^s$.
Now pointwise estimates on Fourier side give

$$\|\pd_x(uvw)\|_{X_{s,-b}}\ls\|D_x^{\frac12}((D_x^{\frac12}\lb \nabla_{xy} \rb^su)vw)\|_{X_{0,-b}}+
\|D_x^{\frac12}((D_x^{\frac12}u)(\lb \nabla_{xy} \rb^sv)w)\|_{X_{0,-b}}+ \dots ,$$

where the dots indicate similar terms. For the first contribution we use first \eqref{ZK.14} and then
\eqref{ZK.13} to obtain the upper bound

$$\|(D_x^{\frac12}\lb \nabla_{xy} \rb^su)v\|_{L^2_{xyt}}\|\lb D_x \rb^s w\|_{X_{0,b}} \ls 
\|u\|_{X_{s,b}}\|v\|_{X_{s,b}}\|w\|_{X_{s,b}}.$$

For the second contribution we start again with \eqref{ZK.14} and continue with H\"older's inequality
and two applications of \eqref{ZK.10} to see that it is bounded by

$$\|(D_x^{\frac12}u)(\lb \nabla_{xy} \rb^sv)\|_{L^2_{xyt}}\|\lb D_x \rb^s w\|_{X_{0,b}} \ls
\|u\|_{X_{s,b}}\|v\|_{X_{s,b}}\|w\|_{X_{s,b}}.$$

Thus we have achieved

\begin{equation}\label{ZK.15}
 \|\pd_x(uvw)\|_{X_{s,-b}}\ls\|u\|_{X_{s,b}}\|v\|_{X_{s,b}}\|w\|_{X_{s,b}},
\end{equation}

where $s,b > \frac12$. It remains to replace the $-b < - \frac12$ on the left by a $b' > -\frac12$.
For that purpose we estimate for $\sigma > \frac54$

\begin{eqnarray*}
 \|\lb \nabla_{xy} \rb^{\sigma} \pd_x(uvw)\|_{L^2_{xyt}} & \ls &
\|(\lb \nabla_{xy} \rb^{\sigma} \pd_xu)vw\|_{L^2_{xyt}} + 
\| (\pd_x u)(\lb \nabla_{xy} \rb^{\sigma}v)w\|_{L^2_{xyt}} + \dots \\
& \ls & \|\lb \nabla_{xy} \rb^{\sigma} \pd_xu\|_{L^{\infty}_xL^2_{yt}}\|v\|_{L^4_xL^{\infty}_{yt}}\|w\|_{L^4_xL^{\infty}_{yt}} \\
& + & \|(\pd_x u)(\lb \nabla_{xy} \rb^{\sigma}v)\|_{L^{4}_xL^2_{yt}}\|w\|_{L^4_xL^{\infty}_{yt}} + \dots
\end{eqnarray*}

For the first contribution we get the upper bound $\|u\|_{X_{\sigma,b}}\|v\|_{X_{\sigma,b}}\|w\|_{X_{\sigma,b}}$
by Kato smoothing for the first and by the maximal function estimate for the second and third factor. The second
contribution is estimated by

$$\|\pd_xu\|_{L^{\infty}_xL^4_{yt}}\|\lb \nabla_{xy} \rb^{\sigma}v\|_{L^{4}_{xyt}}\|w\|_{L^4_xL^{\infty}_{yt}}
\ls \|u\|_{X_{\sigma,b}}\|v\|_{X_{\sigma,b}}\|w\|_{X_{\sigma,b}},$$

where we have used a Sobolev embedding in $x$ and the $L^4$-Strichartz estimate for the first, the same
Strichartz estimate for  the second and the maximal function estimate for the third factor. This shows that for
$b > \frac12$ and $\sigma > \frac54$

\begin{equation}\label{ZK.16}
 \|\lb \nabla_{xy} \rb^{\sigma} \pd_x(uvw)\|_{L^2_{xyt}}\ls 
\|u\|_{X_{\sigma,b}}\|v\|_{X_{\sigma,b}}\|w\|_{X_{\sigma,b}}.
\end{equation}

Finally interpolation among \eqref{ZK.15} and \eqref{ZK.16} gives the claimed estimate.

\end{proof}

\end{document}